\documentclass[11pt,english]{amsart}
\usepackage[utopia]{mathdesign}
\usepackage[a4paper]{geometry}
\usepackage{amstext}
\usepackage{amsthm}
\usepackage{setspace}
\makeatletter

\newcommand{\noun}[1]{\textsc{#1}}

  \theoremstyle{plain}
  \newtheorem{lem}{\protect\lemmaname}
\theoremstyle{plain}
\newtheorem{thm}{\protect\theoremname}
 \theoremstyle{definition}
  \newtheorem{example}{\protect\examplename}
  \theoremstyle{definition}
  \newtheorem{defn}{\protect\definitionname}

\makeatother

\usepackage{babel}
  \providecommand{\definitionname}{Definition}
  \providecommand{\examplename}{Example}
  \providecommand{\lemmaname}{Lemma}
\providecommand{\theoremname}{Theorem}

\begin{document}

\title{On the simultaneous approximation of coefficients of schlicht functions}

\author{Eberhard Michel}
\maketitle
\begin{center}
Sonnenblumenweg 5
\par\end{center}

\begin{center}
65201 Wiesbaden
\par\end{center}
\begin{abstract}
A modified version of the Hardy-Littlewood tauberian theorem is used
to prove under which conditions the moduli of the coefficients $|a(n)|/n$
of schlicht functions tend uniformly to their Hayman-indexes as $n\rightarrow\infty$. 
\end{abstract}

{\small{}AMS Subject Classification}\noun{\small{}: }{\small{}30C50
Coefficient problems of univalent and multivalent functions}\\
{\small{}Keywords: univalent function, Hardy-Littlewood tauberian
theorem, Hayman index, Bazilevich's Theorem, full mapping, schlicht
approximation}{\small \par}

In the sequel let $\mathbb{N}_{0}=\mathbb{N}\cup\{0\}$, let $\mathbb{D}$
denote the unit disk, let $\Delta=\{z\in\mathbb{C}:|z|>1\}$ and let
$S$ denote the set of schlicht functions that are univalent in $\mathbb{D}$.
A function $g:\Delta\rightarrow\mathbb{C}$, $g(z)=z+\sum_{n=0}^{\infty}b_{n}z^{-n}$
analytic and univalent in $\Delta$ is called a full mapping if the
complement of $g(\Delta)$ with respect to $\mathbb{C}$ has two-dimensional
Lebesgue-measure zero, the corresponding class is denoted by $\widetilde{\Sigma}$(for
further details see for instance \cite{1}, chapter 2). Suppose that
$f\in S$ is given by $f(z)=\sum_{n=0}^{\infty}a_{n}z^{n}$ and let
$0\leq\alpha\leq1$ denote its Hayman-index. The Hayman-index $\alpha$
of a schlicht function $f\in S$ is defined by the formula $\alpha=\underset{r\rightarrow1-}{\mathit{lim}}(1-r)^{2}M_{\infty}(r,f)$,
where $M_{\infty}(r,f)=\mathit{max}\{|f(z)|:|z|=r\}$. $f\in S$ is
said to be of slow growth if $\alpha=0$ and it is said to be of maximal
growth if $\alpha>0$. Then Hayman's regularity theorem asserts that
$|a_{n}|/n\rightarrow\alpha$ as $n\rightarrow\infty$, however by
a result of Shirokov(\cite{2}) $(|a_{n}|/n)$ may converge arbitrarily
slowly to $\alpha$ as $n\rightarrow\infty$. So the question arises
under which conditions the terms $|a_{n}|/n$ converge more regularly
to $\alpha$ as $n\rightarrow\infty$. In order to give an answer
to this question families of schlicht functions with certain properties
will be considered in the sequel and the tool mainly used will be
an extension of the Hardy-Littlewood tauberian theorem, which is introduced
in Lemma 1. It's proof will be given here since the author isn't aware
of any reference.
\begin{lem}
(Simultaneous tauberian approximation). Let $(f_{m})$, $f_{m}(z)=\sum_{k=0}^{\infty}a_{k}^{(m)}z^{k}$
denote a sequence of functions analytic in the unit disk with real
coefficients, that is $a_{k}^{(m)}\in\mathbb{R}$ for $k\in\mathbb{N}_{0}$,
$m\in\mathbb{N}$. Furthermore for $m\in\mathbb{N}$ let $s_{n}^{(m)}$,
$n\in\mathbb{N}_{0}$ be defined by $s_{n}^{(m)}=\sum_{k=0}^{n}a_{k}^{(m)}$
and suppose that \\
(i) there exist constants $0<K<\infty$ and $\alpha_{m}\in\mathbb{R}$,
$m\in\mathbb{N}$ such that $\underset{t\rightarrow1-}{\mathit{lim}}f_{m}(t)=\alpha_{m}$
and $|\alpha_{m}|<K$\\
(ii) $(f_{m})$ converges uniformly in $[0,1]$ as $m\rightarrow\infty$\\
(iii) there exists a constant $0<L<\infty$ such that $s_{n}^{(m)}-\alpha_{m}<L$
for every $m\in\mathbb{N}$, $n\in\mathbb{N}_{0}$\\
Then for every $\epsilon>0$ there exists a $N(\epsilon)\in\mathbb{N}$
such that 
\[
\left|\alpha_{m}-\frac{1}{n+1}\sum_{k=0}^{n}(n+1-k)a_{k}^{(m)}\right|=\left|\alpha_{m}-\frac{1}{n+1}\sum_{k=0}^{n}s_{k}^{(m)}\right|<\epsilon
\]
whenever $m,n>N(\epsilon)$.
\end{lem}
\begin{proof}
Let $\epsilon>\text{0}$ and consider $g_{m}(t)=f_{m}(t)-\alpha_{m}=(1-t)\sum_{n=0}^{\infty}(s_{n}^{(m)}-\alpha_{m})t^{n}$,
$m\in\mathbb{N}$. Define $\lambda:[0,1]\rightarrow\mathbb{R}$ by
\[
t\lambda(t)=\begin{cases}
0 & \mathit{if}\;0\leq t<e^{-1}\\
1 & \mathit{if}\;e^{-1}\leq t\leq1
\end{cases}.
\]
Then by the Weierstrass-Approximation-Theorem(see \cite{3}, chapter
7.53) there exist polynomials $p_{\epsilon}$ , $P_{\epsilon}$ such
that $p_{\epsilon}(t)<\lambda(t)<P_{\epsilon}(t)$ and
\begin{eqnarray}
\int_{0}^{1}P{}_{\epsilon}(t)-p_{\epsilon}(t)\,dt & < & \epsilon L^{-1}.
\end{eqnarray}
Now let $P_{\epsilon}(t)-p_{\epsilon}(t)=\sum_{j=0}^{\nu}d_{j}t^{j}$,
$d_{j}\in\mathbb{R}$, $j=0,..,\nu$. Then on one hand
\begin{equation}
(1-t)\,\sum_{k=0}^{\infty}t^{k}\left(P_{\epsilon}(t^{k})-p_{\epsilon}(t^{k})\right)=(1-t)\,\sum_{j=0}^{\nu}d_{j}\sum_{k=0}^{\infty}(t^{j+1})^{k}=\sum_{j=0}^{\nu}d_{j}\frac{1-t}{1-t^{j+1}}
\end{equation}
and on the other hand 
\begin{eqnarray}
\sum_{j=0}^{\nu}d_{j}\frac{1-t}{1-t^{j+1}} & \rightarrow & \sum_{j=0}^{\nu}\frac{d_{j}}{j+1}=\int_{0}^{1}P_{\epsilon}(t)-p_{\epsilon}(t)\,dt
\end{eqnarray}
as $t\rightarrow1$-. Hence by (2) and (3) there exists a $T_{1}(\epsilon)$
with $0<T_{1}(\epsilon)<1$ such that
\begin{equation}
\left|(1-t)\,\sum_{k=0}^{\infty}t^{k}\left(P_{\epsilon}(t^{k})-p_{\epsilon}(t^{k})\right)-\int_{0}^{1}P_{\epsilon}(t)-p_{\epsilon}(t)\,dt\right|<\epsilon L^{-1}
\end{equation}
 if $1\geq t>T_{1}(\epsilon)$. But by (1),(4) and (iii) 
\begin{eqnarray*}
(1-t)\sum_{k=0}^{\infty}(s_{k}^{(m)}-\alpha_{m})\left(t^{k}\lambda(t^{k})-t^{k}p_{\epsilon}(t^{k})\right) & < & (1-t)\sum_{k=0}^{\infty}Lt^{k}\left(\lambda(t^{k})-p_{\epsilon}(t^{k})\right)\\
 & < & (1-t)\,L\sum_{k=0}^{\infty}t^{k}\left(P_{\epsilon}(t^{k})-p_{\epsilon}(t^{k})\right)<2\epsilon
\end{eqnarray*}
and
\begin{eqnarray*}
(1-t)\sum_{k=0}^{\infty}(s_{k}^{(m)}-\alpha_{m})\left(t^{k}P_{\epsilon}(t^{k})-t^{k}\lambda(t^{k})\right) & < & (1-t)\sum_{k=0}^{\infty}Lt^{k}\left(P_{\epsilon}(t^{k})-\lambda(t^{k})\right)\\
 & < & (1-t)\,L\sum_{k=0}^{\infty}t^{k}\left(P_{\epsilon}(t^{k})-p_{\epsilon}(t^{k})\right)<2\epsilon
\end{eqnarray*}
if $1\geq t>T_{1}(\epsilon)$. The last two inequalities imply that
\begin{equation}
(1-t)\sum_{k=0}^{\infty}c_{k}^{(m)}t^{k}P_{\epsilon}(t^{k})-2\epsilon<(1-t)\sum_{k=0}^{\infty}c_{k}^{(m)}t^{k}\lambda(t^{k})<(1-t)\sum_{k=0}^{\infty}c_{k}^{(m)}t^{k}p_{\epsilon}(t^{k})+2\epsilon
\end{equation}
where $c_{k}^{(m)}=s_{k}^{(m)}-\alpha_{m}$, $k\in\mathbb{N}_{0}$,
$m\in\mathbb{N}$ if $1\geq t>T_{1}(\epsilon)$. The next step is
to show that the terms on the left hand side and on the right hand
side of (5) involving $P_{\epsilon}$ and $p_{\epsilon}$ are smaller
than $\epsilon$ if $t\in(0,1)$ is chosen large enough. This will
be shown for an arbitrarily chosen polynomial $P$, so that it will
also hold for $P_{\epsilon}$ and $p_{\epsilon}$. Let $P$ be defined
by $tP(t)=\sum_{j=1}^{\mu}b_{j}t^{j}$ and show that there exist constants
$T(\epsilon)>0$ and $M\in\mathbb{N}$(dependent on $P$ and $\epsilon$)
such that 
\begin{equation}
(1-t)\left|\sum_{k=0}^{\infty}(s_{k}^{(m)}-\alpha_{m})\,t^{k}P(t^{k})\right|<2\epsilon
\end{equation}
whenever $T(\epsilon)<t^{\mu}<1$ and $m>M$. In oder to prove (6)
observe that by hypothesis the sequence $(g_{m})$ converges uniformly
on $[0,1]$. Hence there exists a $M\in\mathbb{N}$ such that for
$0\leq t\leq1$ 
\begin{equation}
\left|g_{M}(t)-g_{m}(t)\right|<\mathit{min}\left\{ \epsilon\left(\sum_{i=1}^{\mu}\left|b_{i}\right|\right)^{-1},\epsilon\right\} 
\end{equation}
if $m>M$. But since $g_{M}$ is continuous on $[0,1]$ by (i) and
since $g_{M}(1)=0$ there exists a $T(\epsilon)>0$ such that 
\begin{equation}
\left|g_{M}(t^{j})\right|<\mathit{min}\left\{ \epsilon\left(\sum_{i=1}^{\mu}\left|b_{i}\right|\right)^{-1},\epsilon\right\} 
\end{equation}
if $T(\epsilon)<t^{\mu}\leq1$ for $j=1,..,\mu$. Consequently by
(8)
\begin{equation}
\sum_{j=1}^{\mu}\left|b_{j}\right|\left|g_{M}(t^{j})\right|<\epsilon
\end{equation}
and (7) implies that 
\begin{equation}
\sum_{j=1}^{\mu}\left|b_{j}\right|\left|g_{M}(t^{j})-g_{m}(t^{j})\right|<\epsilon
\end{equation}
whenever $T(\epsilon)<t^{\mu}\leq1$ and $m>M$. Now, let $t\in(0,1)$
be chosen arbitrarily, let $y_{j}$ be defined by $y_{j}=t^{j}$ for
$j=1,..,\mu$ and observe that the series $\sum_{k=0}^{n}(s_{k}^{(m)}-\alpha_{m})\,y_{j}^{k}$
converge as $n\rightarrow\infty$ for each $m\in\mathbb{N}$ and $j=1,..,\mu$.
Then, by the usual algebra of addition and multiplication of convergent
series(\cite{4},Theorem 3.47),
\begin{eqnarray*}
(1-t)\left|\sum_{k=0}^{\infty}(s_{k}^{(m)}-\alpha_{m})\,t^{k}P(t^{k})\right| & = & (1-t)\left|\sum_{k=0}^{\infty}(s_{k}^{(m)}-\alpha_{m})\sum_{j=1}^{\mu}b_{j}\,(t^{k})^{j}\right|\\
 & = & (1-t)\left|\sum_{k=0}^{\infty}\sum_{j=1}^{\mu}b_{j}(s_{k}^{(m)}-\alpha_{m})\,y_{j}^{k}\right|
\end{eqnarray*}
\begin{eqnarray*}
 & = & (1-t)\left|\sum_{j=1}^{\mu}b_{j}\sum_{k=0}^{\infty}(s_{k}^{(m)}-\alpha_{m})\,y_{j}^{k}\right|\\
 & \leq & \sum_{j=1}^{\mu}\left|b_{j}\right|(1-t^{j})\left|\sum_{k=0}^{\infty}(s_{k}^{(m)}-\alpha_{m})\,(t^{j})^{k}\right|\\
 & = & \sum_{j=1}^{\mu}\left|b_{j}\right|\left|g_{m}(t^{j})\right|.
\end{eqnarray*}
However (9), (10) and the triangle inequality imply that $\sum_{j=1}^{\mu}\left|b_{j}\right|\left|g_{m}(t^{j})\right|<2\epsilon$
whenever $T(\epsilon)<t^{\mu}\leq1$ and $m>M$, which proves (6).
Hence, by (6) there exist constants $0<T_{2}(\epsilon)<1$ and $N_{1}(\epsilon)\in\mathbb{N}$
such that the inequality (6) will hold simultaneously for both $p_{\epsilon}$
and $P_{\epsilon}$ if $t>T_{2}(\epsilon)$ and if $m>N_{1}(\epsilon)$.
And the inequality (5) together with the inequality (6)(applied to
$p_{\epsilon}$ and $P_{\epsilon}$) yields 
\begin{equation}
(1-t)\left|\sum_{k=0}^{\infty}(s_{k}^{(m)}-\alpha_{m})t^{k}\lambda(t^{k})\right|<4\epsilon
\end{equation}
if $\mathit{max\{}T_{1}(\epsilon),T_{2}(\epsilon)\}<t\leq1$ and if
$m,n>N_{1}(\epsilon)$. Finally, in order to complete the proof of
Lemma 1 choose the sequence $t_{n}=1-(n+1)^{-1}$ and observe that
\begin{equation}
(1-\frac{1}{n+1})^{n+1}<e^{-1}<(1-\frac{1}{n+1})^{n}
\end{equation}
 for each $n\in\mathbb{N}$. Then there exists a $N_{2}(\epsilon)\in\mathbb{N}$
such that $\mathit{max\{}T_{1}(\epsilon),T_{2}(\epsilon)\}<t_{n}<1$
if $n>N_{2}(\epsilon)$ and (11), (12) and the definition of $t\lambda(t)$
applied to the sequence $(t_{n})$ imply that
\[
\frac{1}{n+1}\left|\sum_{k=0}^{n}\left(s_{k}^{(m)}-\alpha_{m}\right)\right|<4\epsilon
\]
if $m,n>\mathit{max}\{N_{1}(\epsilon),N_{2}(\epsilon)\}$. 
\end{proof}
The next result will be used several times in the sequel and formulated
as a lemma here.
\begin{lem}
Let $(g_{n})$, $g_{n}:[0,1]\rightarrow\mathbb{R}$ be a sequence
of uniformly bounded functions such that each $g_{n}$ is non-increasing
in $[0,1]$ and suppose that there exists a function $g:[0,1]\rightarrow\mathbb{R}$
such that for each $r\in[0,1]$ $g_{n}(r)\rightarrow g(r)$ as $n\rightarrow\infty$.
Then $\underset{r\rightarrow1-}{\mathit{lim}}g(r)\geq\underset{n\rightarrow\infty}{\mathit{lim}}g_{n}(1)=g(1)$.
Furthermore, if $g:[0,1]\rightarrow\mathbb{R}$ is continuous then
$(g_{n})$ converges uniformly in $[0,1]$ to $g$ as $n\rightarrow\infty.$
\end{lem}
\begin{proof}
The inequality of Lemma 2 is a trivial consequence of properties of
monotonic functions. Let $\epsilon>0$ be chosen arbitrarily. It is
clear that the limit-function $g:[0,1]\rightarrow\mathbb{R}$ is non-increasing
in $[0,1]$. Since $g:[0,1]\rightarrow\mathbb{R}$ is uniformly continuous
there exist real numbers $0=t_{0}<t_{1}<..<t_{m}=1$ such that $0\leq g(t_{k-1})-g(t_{k})<\epsilon$
for each $k=1,..,m$. Furthermore there exists a $N(\epsilon)\in\mathbb{N}$
such that for each $k=0,1,..,m$ $|g_{n}(t_{k})-g(t_{k})|<\epsilon$
whenever $n>N(\epsilon)$. Let $t\in[0,1]$ be chosen arbitrarily,
say $t_{k-1}<t<t_{k}$ for some $1\leq k\leq m$. Then, since $g$
and $g_{n}$, $n\in\mathbb{N}$ are non-inceasing, the two properties
mentioned above yield 
\[
g(t)-g_{n}(t)<\epsilon+g(t_{k-1})-g_{n}(t)\leq\epsilon+g(t_{k-1})-g_{n}(t_{k-1})<2\epsilon
\]
and 
\[
g(t)-g_{n}(t)>g(t_{k})-g_{n}(t)-\epsilon\geq g(t_{k})-g_{n}(t_{k})-\epsilon>-2\epsilon
\]
if $n>N(\epsilon)$. This proves the lemma.
\end{proof}
In order to introduce the first theorem consider the sequence $(f_{m})$,
$f_{m}(z)=r_{m}^{-1}k(r_{m}z)$ where $k$ denotes the Koebe-function
and $(r_{m})$,$m\in\mathbb{N}$ denotes a sequence such that $0<r_{m}<1$
and $r_{m}\rightarrow1$ as $m\rightarrow\infty$. Then each $f_{m}$
of the sequence has Hayman-index $0$, yet the sequence $(f_{m})$
converges locally uniformly in $\mathbb{D}$ to the limit-function
$k$ which has Hayman-index $1$. Therefore, on the one hand $|a_{n}(f_{m})|/n\rightarrow0$
as $n\rightarrow\infty$ for each $f_{m}$ by Hayman's regularity
theorem, but on the other hand $|a_{n}(f_{m})|/n\rightarrow1$ as
$m\rightarrow\infty$ because the sequence $(f_{m})$ converges locally
uniformly in $\mathbb{D}$ to the Koebe-function. So, if the seqence
($f_{m})$, $f_{m}\in S$ converges locally uniformly in $\mathbb{D}$
to a schlicht function $f$ as $m\rightarrow\infty$, the convergence
of the Hayman-indexes $\alpha(f_{m})$ of $f_{m}\in S$ to the Hayman-index
$\alpha(f)$ of the limit-function $f\in S$ seems to be an essential
hypothesis in order to establish the simultaneous convergence of $|a_{n}(f_{m})|/n$
to $\alpha(f_{m})$ as $m,n\rightarrow\infty$. The first theorem
assumes that this minimal hypothesis holds.
\begin{thm}
Let $(f_{m})$, $f_{m}\in S$, $m\in\mathbb{N}$ be given by $f_{m}(z)=\sum_{n=1}^{\infty}a_{n}^{(m)}z^{n}$
and let $\epsilon>0$ be chosen arbitrarily. Furthermore suppose that\\
(i) $(f_{m})$ converges locally uniformly in $\mathbb{D}$ to a schlicht
function $f$ as $m\rightarrow\infty$ \\
(ii) $\alpha_{m}\rightarrow\alpha$ as $m\rightarrow\infty$ where
$\alpha_{m}$ and $\alpha$ denote the Hayman-indexes of $f_{m}$
and $f$ respectively\\
Then there exists a constant $N(\epsilon)$ only dependent on $\epsilon$
such that 
\[
\frac{\left|a_{n}^{(m)}\right|}{n}<\epsilon+\sqrt{\alpha_{m}}
\]
whenever $m,n>N(\epsilon)$.
\end{thm}
\begin{proof}
Let $\epsilon>0$ be chosen arbitrarily and consider the case $\alpha>0$,
that is the limit function $f\in S$ is of maximal growth, first.
Then, without loss in generality, it may be supposed that for each
$m\in\mathbb{N}$ $f_{m}\in S$ has Hayman-index $\alpha_{m}>0$ and
radius of greatest growth in the direction of the positive real axis.
For $m\in\mathbb{N}$ and $0<r<1$ define $h_{m}$ by $h_{m}(r)=\mathit{log}((1-r)^{2}r^{-1}f_{m}(r))=\sum_{k=1}^{\infty}2(\gamma_{k}^{(m)}-k^{-1})r^{k}$,
where $\gamma_{k}^{(m)}$, $k\in\mathbb{N}$ denote the logarithmic
coefficients of $f_{m}\in S$. Then for each $m\in\mathbb{N}$ $\mathit{Re}\{h_{m}(r)\}=\mathit{log}((1-r)^{2}r^{-1}|f_{m}(r)|)$
is non-increasing in $[0,1]$(\cite{1},chapter 5.5), and by (ii)
the limit function defined by $\underset{m\rightarrow\infty}{\mathit{lim}}\mathit{Re}\{h_{m}(r)\}=\mathit{log}((1-r)^{2}r^{-1}|f(r)|)$
for $r\in[0,1]$ is continuous in $[0,1]$. Hence, by Lemma 2, the
sequence $(\mathit{Re}\,h_{m})$ converges uniformly in $[0,1]$ as
$m\rightarrow\infty$. Furthermore by Milin's lemma(\cite{1}, chapter
5.4) 
\[
\sum_{k=1}^{n}2\left(\mathit{Re}\gamma_{k}^{(m)}-\frac{1}{k}\right)=\sum_{k=1}^{n}\left(k\left|\gamma_{k}^{(m)}\right|^{2}-\frac{1}{k}\right)-\sum_{k=1}^{n}k\left|\gamma_{k}^{(m)}-\frac{1}{k}\right|^{2}\leq\delta<0.312
\]
where $\delta$ denotes Milin's constant. Therefore Lemma 1 can be
applied and asserts that there exists a $N_{1}(\epsilon)\in\mathbb{N}$
such that
\begin{equation}
\left|\frac{2}{n+1}\sum_{k=1}^{n}(n+1-k)\left(\mathit{Re}\gamma_{k}^{(m)}-\frac{1}{k}\right)-\mathit{log}\alpha_{m}\right|<\epsilon
\end{equation}
whenever $m,n>N_{1}(\epsilon)$. Since by Basilevich's theorem(\cite{1},Theorem
5.5) 
\begin{equation}
\frac{1}{n+1}\sum_{k=1}^{n}(n+1-k)k\left|\gamma_{k}^{(m)}-\frac{1}{k}\right|^{2}\leq\sum_{k=1}^{\infty}k\left|\gamma_{k}^{(m)}-\frac{1}{k}\right|^{2}\leq-\frac{1}{2}\mathit{log}\alpha_{m}
\end{equation}
for each $m\in\mathbb{N}$ (13) and (14) imply that 
\begin{equation}
\frac{1}{n+1}\sum_{k=1}^{n}(n+1-k)\left(k\left|\gamma_{k}^{(m)}\right|^{2}-\frac{1}{k}\right)<-\frac{1}{2}\mathit{log}\alpha_{m}+\mathit{log}\alpha_{m}+\epsilon
\end{equation}
if $m,n>N_{1}(\epsilon)$. But the second Lebedev-Milin inequality
applied to (15) yields 
\[
\left|a_{n+1}^{(m)}\right|\leq\sum_{k=0}^{n}\left|b_{k}^{(m)}\right|^{2}\leq(n+1)\mathit{exp}\left\{ \frac{1}{n+1}\sum_{k=1}^{n}(n+1-k)\left(k\left|\gamma_{k}^{(m)}\right|^{2}-\frac{1}{k}\right)\right\} <(n+1)e^{\epsilon}\sqrt{\alpha_{m}}
\]
whenever $m,n>N_{1}(\epsilon)$ where $b_{k}^{(m)}$, $m\in\mathbb{N}$,
$k\in\mathbb{N}_{0}$ denote the coefficients of the functions $\sqrt{r^{-1}f_{m}(r)}$,
$m\in\mathbb{N}$. This proves the theorem, if the limit function
$f\in S$ has Hayman-index $\alpha>0$. In order to complete the proof,
consider the second case, that is suppose that the limit-function
$f\in S$ has Hayman-index $\alpha=0$. For $m\in\mathbb{N}$ define
$g_{m}(r)=r^{-1}(1-r)^{2}M_{\infty}(r,f_{m})$ and $g(r)=r^{-1}(1-r)^{2}M_{\infty}(r,f)$.
Then each of the functions $g,g_{m}:[0,1]\rightarrow\mathbb{R}$,
$m\in\mathbb{N}$ is non-increasing and non-negative(\cite{1}, chapter
5.5 Lemma). Moreover they are continuous in $[0,1]$ since the functions
$M_{\infty}(.,f)$ and $M_{\infty}(.,f_{m})$, $m\in\mathbb{N}$ clearly
are continuous in $[0,1)$. They are also continuous at $r=1$ because
$g(r)\rightarrow\alpha$ as $r\rightarrow1-$ and $g_{m}(r)\rightarrow\alpha_{m}$
as $r\rightarrow1-$ for $m\in\mathbb{N}$. In order to show that
$g_{m}(r)\rightarrow g(r)$ as $m\rightarrow\infty$ if $r\in[0,1]$
observe that this holds for $r=1$ and $r=0$ because of hypothesis
(ii) and since for every $m\in\mathbb{N}$ $g_{m}(0)=g(0)=1$. So
let $r\in(0,1)$ and suppose on the contrary that $g_{m}(r)\nrightarrow g(r)$
as $m\rightarrow\infty$ for some $r\in(0,1)$. Then there exists
a subsequence $(g_{m(k)}(r))$ and a $\beta\in[0,1]$ such that
\begin{equation}
\underset{k\rightarrow\infty}{\mathit{lim}}g_{m(k)}(r)=\beta\neq g(r).
\end{equation}
However the subsequence can be chosen in such a way that $M_{\infty}(r,f_{m(k)})=|f_{m(k)}(z_{m(k)})|$
where $z_{m(k)}\rightarrow z$ as $k\rightarrow\infty$ for some $z$,
$z_{m(k)}$ with $|z|\leq r$, $|z_{m(k)}|\leq r$ , $k\in\mathbb{N}$,
and such that (16) holds. But, because of the uniform convergence
of the sequence $(f_{m})$ in $|z|\leq r$, there exists a $N_{1}(\epsilon)\in\mathbb{N}$
such that 
\begin{equation}
\left|\left|f_{m(k)}(z_{m(k)})\right|-\left|f(z)\right|\right|=\left|M_{\infty}(r,f_{m(k)})-\left|f(z)\right|\right|<\frac{\epsilon}{3}.
\end{equation}
if $k>N_{1}(\epsilon)$. Let $w$, $|w|\leq r$ be a point such that
$|f(w)|=M_{\infty}(r,f)$. Then, since $(f_{m})$ converges uniformly
in $|z|\leq r$ as $m\rightarrow\infty$, there exists a $N_{2}(\epsilon)\in\mathbb{N}$
such that 
\begin{equation}
\left|\left|f_{m(k)}(w)\right|-M_{\infty}(r,f)\right|<\frac{\epsilon}{3}
\end{equation}
if $k>N_{2}(\epsilon)$. Therefore (17), (18) and the fact, that $f$
and $f_{m(k)}$, $k\in\mathbb{N}$ take on their maxima if $|z|\leq r$
at $z=w$ and $z=z_{m(k)}$, $k\in\mathbb{N}$ respectively, yield
\begin{equation}
0\leq M_{\infty}(r,f_{m(k)})-\left|f_{m(k)}(w)\right|\leq M_{\infty}(r,f_{m(k)})-\left|f_{m(k)}(w)\right|+M_{\infty}(r,f)-\left|f(z)\right|<\frac{2\epsilon}{3}
\end{equation}
whenever $k>\mathit{max}\{N_{1}(\epsilon),N_{2}(\epsilon)\}$. Now
(18) and (19 ) imply that 
\[
\left|M_{\infty}(r,f_{m(k)})-M_{\infty}(r,f)\right|\leq\left|M_{\infty}(r,f_{m(k)})-\left|f_{m(k)}(w)\right|\right|+\left|\left|f_{m(k)}(w)\right|-M_{\infty}(r,f)\right|<\epsilon
\]
if $k>\mathit{max}\{N_{1}(\epsilon),N_{2}(\epsilon)\}$ which proves
that $g_{m(k)}(r)\rightarrow g(r)$ as $k\rightarrow\infty$, a contradiction
to (16). So $g_{m}(r)\rightarrow g(r)$ as $m\rightarrow\infty$ for
$r\in[0,1]$ and an appeal to Lemma 2 shows that $g_{m}\rightarrow g$
uniformly on $[0,1]$ as $m\rightarrow\infty$. But by Prawitz' theorem(\cite{1},Theorem
2.22) 
\[
\frac{d}{dr}\left\{ \frac{1}{2\pi}\int_{0}^{2\pi}\left|f_{m}(re^{it})\right|dt\right\} \leq\frac{1}{r}M_{\infty}(r,f_{m})=g_{m}(r)\frac{1}{(1-r)^{2}}
\]
if $m\in\mathbb{N}$. Integrating the last inequality from $r_{0}$
to $r$ where $0<r_{0}<r<1$ yields
\begin{eqnarray*}
\frac{1}{2\pi}\int_{0}^{2\pi}\left|f_{m}(re^{it})\right|dt-\frac{1}{2\pi}\int_{0}^{2\pi}\left|f_{m}(r_{0}e^{it})\right|dt & \leq & g_{m}(r_{0})\int_{r_{0}}^{r}(1-t)^{-2}dt\\
 & = & g_{m}(r_{0})\left(\frac{1}{1-r}-\frac{1}{1-r_{0}}\right).
\end{eqnarray*}
It follows that 
\begin{equation}
(1-r)\frac{1}{2\pi}\int_{0}^{2\pi}\left|f_{m}(re^{it})\right|dt\leq(1-r)\frac{1}{2\pi}\int_{0}^{2\pi}\left|f_{m}(r_{0}e^{it})\right|dt+g_{m}(r_{0})
\end{equation}
for $m\in\mathbb{N}$. Because $g_{m}\rightarrow g$ uniformly on
$[0,1]$ as $m\rightarrow\infty$, because each $g_{m}$ is continuous
on $[0,1]$ and because $g(1)=0$ there exists a $N_{3}(\epsilon)\in\mathbb{N}$
such that
\begin{equation}
0<g_{m}(r)<\epsilon
\end{equation}
whenever $m>N_{3}(\epsilon)$ and $1-N_{3}(\epsilon)^{-1}\leq r<1$.
Therefore, if $r_{0}$ in (20) is chosen as $r_{0}=1-N_{3}(\epsilon)^{-1}$
then, since $f_{m}\rightarrow f$ locally uniformly in $\mathbb{D}$
as $m\rightarrow\infty$, there exists a $N_{4}(\epsilon)\in\mathbb{N}$,
$N_{4}(\epsilon)>N_{3}(\epsilon)$ such that for $t\in[0,2\pi]$ $|f_{m}(r_{0}\mathit{exp}(it))-f(r_{0}\mathit{exp}(it))|<\epsilon$
if $m>N_{4}(\epsilon)$. Then there exists also a $N_{5}(\epsilon)\in\mathbb{N}$,
$N_{5}(\epsilon)>N_{4}(\epsilon)$ such that
\begin{eqnarray*}
(1-r)\frac{1}{2\pi}\int_{0}^{2\pi}\left|f_{m}\left(r_{0}e^{it}\right)\right|dt & < & (1-r)\left\{ \epsilon+\frac{1}{2\pi}\int_{0}^{2\pi}\left|f\left(r_{0}e^{it}\right)\right|dt\right\} \\
 & < & \epsilon
\end{eqnarray*}
whenever $1-N_{5}(\epsilon)^{-1}\leq r<1$ and $m>N_{5}(\epsilon)$.
In order to complete the proof of Theorem 1 let $r=r_{n}=1-n^{-1}$
and observe that by (20), (21) and the last inequality the Cauchy
inequality for the coefficients yields
\begin{eqnarray*}
\left|\frac{a_{n}^{(m)}}{n}\right|=\left|a_{n}^{(m)}\right|(1-r_{n}) & \leq & r_{n}^{-n}(1-r_{n})\frac{1}{2\pi}\int_{0}^{2\pi}\left|f_{m}(r_{n}e^{it})\right|dt
\end{eqnarray*}
\begin{eqnarray*}
 & \leq & r_{n}^{-n}\left((1-r_{n})\frac{1}{2\pi}\int_{0}^{2\pi}\left|f_{m}(r_{0}e^{it})\right|dt+g_{m}(r_{0})\right)<4(\epsilon+\epsilon)
\end{eqnarray*}

if $m,n>\mathit{max}\{N_{1}(\epsilon),N_{2}(\epsilon),N_{5}(\epsilon)\}$.
This completes the proof of Theorem 1 if $\alpha=0$.
\end{proof}
There are some interesting applications of Theorem 1 to asymptotic
extremal problems concerning the class of schlicht functions.
\begin{example}
(Asymptotic Bieberbach Conjecture) If for each $n\in\mathbb{N}$ $f_{n}$
is assumed to be a schlicht function that maximizes the modulus of
the $n$-th coefficient then Theorem 1 reveals that in order to prove
the asymptotic Bieberbach conjecture(now superseded by de Branges
theorem) it only has to be shown that for any subsequence $(f_{n(k)})$
of $(f_{n})$ that converges locally uniformly in $\mathbb{D}$ to
some schlicht function $f$ the Hayman-indexes $\alpha(f_{n(k)})$
of $f_{n(k)}\in S$ converge to the Hayman-index $\alpha(f)$ of $f\in S$
as $k\rightarrow\infty$(for more details about the asymptotic Bieberbach
conjecture see \cite{1}, chapter 2.12).
\end{example}
Here is another example how Theorem 1 can be applied to asymptotic
extremal problems for schlicht functions.
\begin{example}
(Asymptotic Zalcman Conjecture) Define functionals $F_{n}$, $n\in\mathbb{N}$
for $f\in S$ by $F_{n}(f)=|a_{n}(f)^{2}-a_{2n-1}(f)|$. Then Zalcman's
conjecture asserts that $F_{n}(f)\leq(n-1)^{2}$ for $f\in S$. Suppose
that $f_{n}$, $n\in\mathbb{N}$ maximizes $F_{n}$ within the class
of schlicht functions and that there were a subsequence $(f_{n(k)})$
that converges to a schlicht function $f$ of slow growth. Then the
sequence $(f_{n(k)})$ satisifies hypothesis (ii) of Theorem 1(by
Lemma 2) and consequently $(n(k)-1)^{-2}F_{n(k)}(f_{n(k)})\rightarrow0$
as $k\rightarrow\infty$ by Theorem 1, a contradiction since the Koebe-function
$k$ satisfies $F_{n}(k)=(n-1)^{2}>0$ if $n\geq2$. This argument
shows that the functionals $(F_{n})$(and certain other sequences
of coefficient functionals), or rather their extremal functions, have
no accumulation points of slow growth(with respect to the topology
of uniform convergence on compact subsets of $\mathbb{D}$). 
\end{example}
The next lemma will be used in the proof of Theorem 2, however it
is of interest in itself, since it provides an extension of an earlier
result of Bazilevich(\cite{5}) concerning the case of equality in
his theorem.
\begin{lem}
Let $f\in S$ have Hayman-index $\alpha>0$ and suppose that $g:\Delta\rightarrow\mathbb{C}$
defined by $g(z)=f(z^{-1})^{-1}$ is a full mapping, that is $g\in\widetilde{\Sigma}$.
Furthermore let $\gamma_{n}$, $n\in\mathbb{N}$ denote the logarithmic
coefficients of $f\in S$ and let $\mathit{exp\,}(i\theta)$, $\theta\in[0,2\pi)$
denote the direction of greatest growth of $f\in S$. Then 
\[
\sum_{n=1}^{\infty}n\left|\gamma_{n}-\frac{1}{n}\mathit{exp\,}(-in\theta)\right|^{2}=-\frac{1}{2}\mathit{log}\alpha
\]
\end{lem}
\begin{proof}
In order to prove the lemma it obviously suffices to consider the
case that $f\in S$ is not a rotation of the Koebe-function. Let $A_{n}(z)=\sum_{k=1}^{\infty}\gamma_{nk}z^{k}$
, $z\in\mathbb{D}$ where $\gamma_{nk}$, $k,n\in\mathbb{N}$ denote
the Grunsky coefficients of $g\in\widetilde{\Sigma}$. For $n\in\mathbb{N}$
and an arbitrarily chosen $\mathbf{u}=(u_{1},u_{2},....)\in\ell^{2}$
define a sequence of continuous linear mappings $T_{n}:\ell^{2}\rightarrow\ell^{2}$
$T_{n}\mathbf{u}=\mathbf{v}=(v_{1},v_{2},.....)$ by $v_{k}=\sqrt{k}\sum_{j=1}^{n}\sqrt{j}\gamma_{kj}u_{j}$,
$k\in\mathbb{N}$. Further, let $\left\Vert .\right\Vert $ and $(.,.)$
denote the norm and inner product of the Hilbert space $\ell^{2}$.
The mappings $T_{n}$, $n\in\mathbb{N}$ are well defined because
the strong Grunsky inequality(\cite{1}, chapter 4.3 formula (10))
yields 
\begin{equation}
\left\Vert T_{n}\mathbf{u}\right\Vert ^{2}=(T_{n}\mathbf{u},T_{n}\mathbf{u})=\sum_{k=1}^{\infty}\left|v_{k}\right|^{2}=\sum_{k=1}^{\infty}k\left|\sum_{j=1}^{n}\gamma_{kj}\sqrt{j}u_{j}\right|^{2}\leq\sum_{j=1}^{n}\left|u_{j}\right|^{2}\leq\left\Vert \mathbf{u}\right\Vert ^{2}
\end{equation}
that is $T_{n}\mathbf{u}\in\ell^{2}$ and $\left\Vert T_{n}\right\Vert \leq1$
for $n\in\mathbb{N}$. Now define a dense subset $U\subseteq\ell^{2}$
by $U=\bigcup_{k\in\mathbb{N}}U_{k}$ where $U_{k}=\{\mathbf{u}=(u_{1},u_{2},....)\in\ell^{2}:u_{j}=0\;\mathit{for}\;j>k\}$,
$k\in\mathbb{N}$ and observe that the sequences $(T_{n}\mathbf{u})$
converge pointwise for each $\mathbf{u}\in U$ as $n\rightarrow\infty$.
Consequently the Banach-Steinhaus theorem(\cite{6}, Theorem 2.7)
implies that the mapping $T:\ell^{2}\rightarrow\ell^{2}$ $T\mathbf{u}=\underset{n\rightarrow\infty}{\mathit{lim}}T_{n}\mathbf{u}$
is well-defined for each $\mathbf{u}\in\ell^{2}$ and is linear and
continuous. Since the vectors $\mathbf{e_{k}}=(\delta_{1,k},\delta_{2,k},...)\in\ell^{2}$,
$k\in\mathbb{N}$, where $\delta_{i,k}=1$ if $i=k$ and $\delta_{i,k}=0$
if $i\neq k$ define an orthonormal basis of $\ell^{2}$ and since
pointwise convergence implies weak convergence Parseval's theorem
yields 
\begin{eqnarray}
\infty>\left\Vert T\mathbf{u}\right\Vert ^{2} & = & \sum_{k=1}^{\infty}\left|(T\mathbf{u},\mathbf{e_{k}})\right|^{2}\nonumber \\
 & = & \sum_{k=1}^{\infty}\left|\underset{n\rightarrow\infty}{\mathit{lim}}(T_{n}\mathbf{u},\mathbf{e_{k}})\right|^{2}\nonumber \\
 & = & \sum_{k=1}^{\infty}\left|\underset{n\rightarrow\infty}{\mathit{lim}}\sqrt{k}\sum_{j=1}^{n}\gamma_{kj}\sqrt{j}u_{j}\right|^{2}\nonumber \\
 & = & \sum_{k=1}^{\infty}k\left|\sum_{j=1}^{\infty}\gamma_{kj}\sqrt{j}u_{j}\right|^{2}
\end{eqnarray}
for each $\mathbf{u}\in\ell^{2}$. On the other hand, since $g\in\widetilde{\Sigma}$,
equality holds in the strong Grunsky inequalities and hence equality
holds in (22), which means that 
\begin{equation}
\infty>\left\Vert T\mathbf{u}\right\Vert ^{2}=\underset{n\rightarrow\infty}{\mathit{lim}}\left\Vert T_{n}\mathbf{u}\right\Vert ^{2}=\underset{n\rightarrow\infty}{\mathit{lim}}\sum_{j=1}^{n}\left|u_{j}\right|^{2}=\left\Vert \mathbf{u}\right\Vert ^{2}.
\end{equation}
Therefore, if $u_{j}=z(\sqrt{j})^{-1}$, $j\in\mathbb{N}$, by the
definition of $A_{n}(z)$ and by (23) and (24)
\begin{equation}
\sum_{n=1}^{\infty}n\left|A_{n}(z)\right|^{2}=\sum_{k=1}^{\infty}k\left|\sum_{j=1}^{\infty}\gamma_{kj}z^{j}\right|^{2}=\sum_{j=1}^{\infty}\frac{1}{j}\left|z_{j}\right|^{2}=-\mathit{log}(1-\left|z\right|^{2})
\end{equation}
if $g\in\widetilde{\Sigma}$. Now let $w(r)=r\mathit{exp\,}\mathit{(i\theta)}$,
$0<r<1$, then $|w(r)|=r$. By (25) and since $\sum_{n=1}^{\infty}nA_{n}(z)z^{n}=\mathit{2log(z^{-1}}f(z))-\mathit{log}f'(z)$
for $z\in\mathbb{D}$(see \cite{1}, proof of Theorem 5.5) it follows
that 
\begin{eqnarray}
\sum_{n=1}^{\infty}n\left|A_{n}(w(r))-\frac{1}{n}\overline{w(r})^{n}\right|^{2} & = & \sum_{n=1}^{\infty}n\left|A_{n}(w(r))\right|^{2}-2\mathit{Re}\left\{ \sum_{n=1}^{\infty}A_{n}(w(r))w(r)^{n}\right\} +\nonumber \\
 &  & \sum_{n=1}^{\infty}\frac{1}{n}\left|w(r)\right|^{2n}\nonumber \\
 & = & -2\mathit{log}\left(1-\left|w(r)\right|^{2}\right)+2\mathit{log}\frac{\left|w(r)^{2}f'(w(r))\right|}{\left|f(w(r))^{2}\right|}\nonumber \\
 & = & \mathit{2log}\left(\frac{\left|f'(w(r))\right|}{1+r}(1-r)^{3}\right)-4\mathit{log}\left(\frac{\left|f(w(r))\right|}{r}(1-r)^{2}\right).
\end{eqnarray}
In order to complete the proof it has to be shown that 
\begin{equation}
\underset{r\rightarrow1-}{\mathit{lim}}\frac{\left|f'(w(r))\right|}{1+r}(1-r)^{3}=\alpha
\end{equation}
if $\alpha>0$. On the one hand, if $0<r<1$, by \cite{1}, Theorem
2.7 
\begin{equation}
\frac{\left|f'(w(r))\right|(1-r)^{3}}{1+r}\leq\frac{\left|f(w(r))\right|(1-r)^{2}}{r}\leq1,
\end{equation}
and since $\alpha<1$ by \cite{1}, chapter 2.3, p. 33 
\begin{equation}
\frac{\partial}{\partial r}\mathit{log}\left(\frac{(1-r)^{3}}{1+r}\left|f'(w(r))\right|\right)=\frac{\partial}{\partial r}\mathit{log}\left|f'(w(r))\right|-\frac{4+2r}{1-r^{2}}<0.
\end{equation}
Hence by (29) $|f'(w(r))|(1-r)^{3}(1+r)^{-1}$ is strictly decreasing,
by (28) it is bounded and consequently $|f'(w(r))|(1-r)^{3}(1+r)^{-1}\rightarrow\vartheta$
as $r\rightarrow1-$ for some $0\leq\vartheta\leq1$. But since $\left|f(w(r))\right|r^{-1}(1-r)^{2}\rightarrow\alpha$
as $r\rightarrow1-$ (28) implies that
\begin{equation}
0\leq\vartheta\leq\alpha.
\end{equation}
 On the other hand by the fundamental theorem of calculus and by de
L'Hospital's rule(\cite{4}, Theorem 5.13)
\begin{eqnarray*}
\alpha=\underset{r\rightarrow1-}{\mathit{lim}}(1-r)^{2}\left|f(w(r))\right| & \leq & \underset{r\rightarrow1-}{\mathit{lim}}(1-r)^{2}\int_{0}^{r}\left|f'\left(t\mathit{exp\,}(i\theta)\right)\right|dt\\
 & = & \underset{r\rightarrow1-}{\mathit{lim}}\frac{\frac{d}{dr}\int_{0}^{r}\left|f'\left(t\mathit{exp\,}(i\theta)\right)\right|dt}{\frac{d}{dr}\{(1-r)^{-2}\}}\\
 & = & \underset{r\rightarrow1-}{\mathit{lim}}\frac{1}{2}(1-r)^{3}\left|f'(w(r))\right|=\vartheta.
\end{eqnarray*}
However then, by (30), $\alpha=\vartheta$, which proves (27). To
complete the proof of the lemma observe that
\begin{equation}
\underset{r\rightarrow1-}{\mathit{lim}}A_{n}\left(r\mathit{exp\,}(i\theta)\right)=2\gamma_{n}-\frac{1}{n}\mathit{exp\,}(-in\theta)
\end{equation}
(see \cite{1}, proof of Theorem 5.5). Now let $r\rightarrow1-$ on
both sides of equation (26) then by (27), (31) and by definition of
the Hayman-index
\begin{eqnarray*}
4\sum_{n=1}^{\infty}n\left|\gamma_{n}-\mathit{exp\,}(-in\theta)\frac{1}{n}\right|^{2} & = & \underset{r\rightarrow1-}{\mathit{lim}}\left\{ 2\mathit{log}\left(\frac{\left|f'(w(r))\right|}{1+r}(1-r)^{3}\right)-4\mathit{log}\left(\frac{\left|f'(w(r))\right|}{r}(1-r)^{2}\right)\right\} \\
 & = & \mathit{2log}\alpha-4\mathit{log}\alpha.
\end{eqnarray*}
This proves the lemma.
\end{proof}
Theorem 1 already provides a uniform convergence result for the coefficients
if the limit function is of slow growth. With the extended version
of Bazilevich's theorem at hand this result can be extended to the
case that the limit function is of maximal growth. 
\begin{thm}
Let $(f_{m})$, $f_{m}\in S$ be given by $f_{m}(z)=\sum_{n=1}^{\infty}a_{n}^{(m)}z^{n}$
and suppose that\\
(i) $(f_{m})$ converges locally uniformly in $\mathbb{D}$ to a schlicht
function $f$ of maximal growth as $m\rightarrow\infty$ such that
the function $z\rightarrow f(z^{-1})^{-1}$, $z\in\Delta$ is a full
mapping \\
(ii) $\alpha_{m}\rightarrow\alpha$ as $m\rightarrow\infty$ where
$\alpha_{m}$ and $\alpha$ denote the Hayman-indexes of $f_{m}$
and $f$ respectively\\
Then for every $\epsilon>0$ there exists a constant $N(\epsilon)$
only dependent on $\epsilon$ such that 
\[
\left|\frac{\left|a_{n}^{(m)}\right|}{n}-\alpha_{m}\right|<\epsilon
\]
whenever $m,n>N(\epsilon)$.
\end{thm}
\begin{proof}
Let $0<\epsilon<\alpha$ be given arbitrarily and suppose that the
limit function $f\in S$ is given by $f(z)=z+\sum_{n=2}^{\infty}a_{n}z^{n}$,
$z\in\mathbb{D}$. Since the limit function $f\in S$ was supposed
to be of maximal growth type, by (ii) there exists a constant $N_{1}(\epsilon)\in\mathbb{N}$
such that 
\begin{equation}
0<\alpha-\epsilon<\alpha_{m}<\alpha+\epsilon
\end{equation}
whenever $m>N_{1}(\epsilon)$. Hence, without loss in generality,
it may be supposed, that each $f_{m}$, $m\in\mathbb{N}$ has Hayman-index
$\alpha_{m}>0$ and radius of greatest growth in the direction of
the positive real axis. For each $m\in\mathbb{N}$ define $h_{m}:\mathbb{D}\rightarrow\mathbb{C}$
by $h_{m}(z)=\sum_{n=1}^{\infty}n|\gamma_{n}^{(m)}-n^{-1}|^{2}z^{n}$
and $g,g_{m}:\mathbb{D}\rightarrow\mathbb{C}$ by $g_{m}(z)=\mathit{log}\{z^{-1}f_{m}(z)\}=2\sum_{n=1}^{\infty}\gamma_{n}^{(m)}z^{n}$
and $g(z)=\mathit{log}\{z^{-1}f(z)\}=2\sum_{n=1}^{\infty}\gamma_{n}z^{n}$
respectively. Here $\gamma_{n}^{(m)}$, $n\in\mathbb{N}$ and $\gamma_{n}$,
$n\in\mathbb{N}$ denote the logarithmic coefficients of $f_{m}\in S$
and $f\in S$ respectively. The strategy of proof is to first apply
Lemma 1 to the sequence $(h_{m})$. In order to do that observe that
by (32) and by Bazilevich's theorem(\cite{1}, Theorem 5.5) the family
$\{h_{m}\}$ is uniformly bounded in $\overline{\mathbb{D}}$ and
therefore there exists a subsequence $(h_{m(k)})$ and an analytic
function $h:\mathbb{D}\rightarrow\mathbb{C}$ such that $h_{m(k)}\rightarrow h$
locally uniformly in $\mathbb{D}$ as $k\rightarrow\infty$. Let $h(z)=\sum_{k=1}^{\infty}\beta_{k}z^{k}$
and consider an arbitrarily chosen coefficient $\beta_{n}$, $n\in\mathbb{N}$.
Then $n|\gamma_{n}^{(m(k))}-n^{-1}|^{2}\rightarrow\beta_{n}$ as $k\rightarrow\infty$
because $h_{m(k)}\rightarrow h$ locally uniformly in $\mathbb{D}$
as $k\rightarrow\infty$. It is well-known that the coefficients $\gamma_{n}^{(m)}$,
$m\in\mathbb{N}$ satisfy the equations $a_{n}^{(m)}=n\gamma_{n}^{(m)}+\sum_{k=1}^{n-1}k\gamma_{k}^{(m)}a_{n-k+1}^{(m)}$
for $m\in\mathbb{N}$. By hypothesis (i) $a_{n}^{(m)}\rightarrow a_{n}$
as $m\rightarrow\infty$for each $n\in\mathbb{N}$ and inductively(by
$n$) it follows that 
\[
\underset{m\rightarrow\infty}{\mathit{lim}}\gamma_{n}^{(m)}=\underset{m\rightarrow\infty}{\mathit{lim}}\frac{1}{n}\left(a_{n}^{(m)}-\sum_{k=1}^{n-1}k\gamma_{k}^{(m)}a_{n-k+1}^{(m)}\right)=\frac{1}{n}\left(a_{n}-\sum_{k=1}^{n-1}k\gamma_{k}a_{n-k+1}\right)=\gamma_{n}.
\]
Consequently also $n|\gamma_{n}^{(m(k))}-n^{-1}|^{2}\rightarrow n|\gamma_{n}-n^{-1}|^{2}$
as $k\rightarrow\infty$ and by the identity theorem $\beta_{n}=n|\gamma_{n}-n^{-1}|^{2}$
for each $n\in\mathbb{N}$. Because the subsequence $(h_{m(k)})$
was chosen arbitrarily each subsequence $(h_{m(k)})$ of the sequence
$(h_{m})$ will converge locally uniformly in $\mathbb{D}$ to $h$
as $k\rightarrow\infty$ and by a result of Montel(\cite{7}, Theorem
2.4.2) the whole sequence $(h_{m})$ converges to the limit-function
$h$ locally uniformly in $\mathbb{D}$ as $m\rightarrow\infty$.
Clearly $-h$ and $-h_{m}$, $m\in\mathbb{N}$ are non-increasing
in $[0,1]$ and by Bazilevich's theorem and Abel's limit theorem $-h$
and $-h_{m}$, $m\in\mathbb{N}$ are continuous in $[0,1]$. In order
to show that $h_{m}(1)\rightarrow h(1)$ as $m\rightarrow\infty$
consider an arbitrarily chosen accumulation point of the sequence
$(h_{m}(1))$, that is consider an arbitrarily chosen subsequence
$(h_{m(k)}(1)$) such that $h_{m(k)}(1)\rightarrow\lambda$ as $k\rightarrow\infty$
for some $\lambda\geq0$. Then on the one hand by Lemma 2 and hypothesis
(ii)
\begin{equation}
\underset{r\rightarrow1-}{\mathit{lim}}\{-h(r)\}\geq\underset{k\rightarrow\infty}{\mathit{lim}}\{-h_{m(k)}(1)\}\geq\underset{k\rightarrow\infty}{\mathit{lim}}\frac{1}{2}\mathit{log}\alpha_{m(k)}=\frac{1}{2}\mathit{log}\alpha
\end{equation}
and on the other hand $h(1)=-(1/2)\mathit{log}\alpha$ by Lemma 3
and Abel's limit theorem. Consequently equality holds in (33) for
each convergent subsequence $(h_{m(k)}(1))$ of the sequence $(h_{m}(1))$
and therefore $h_{m}(1)\rightarrow h(1)$ as $m\rightarrow\infty$.
Now Lemma 2 can be applied and implies that $h_{m}\rightarrow h$
uniformly in $[0,1]$ as $m\rightarrow\infty$ and by Lemma 1 applied
to the functions $h_{m}$ there exists a $N_{2}(\epsilon)\in\mathbb{N}$
such that
\begin{equation}
0\leq-\frac{1}{2}\mathit{log}\alpha_{m}-\sum_{k=1}^{n}k\left|\gamma_{k}^{(m)}-\frac{1}{k}\right|^{2}\leq-\frac{1}{2}\mathit{log}\alpha_{m}-\frac{1}{n+1}\sum_{k=1}^{n}(n+1-k)k\left|\gamma_{k}^{(m)}-\frac{1}{k}\right|^{2}<\epsilon
\end{equation}
if $m,n>N_{2}(\epsilon)$ . By subtracting the two inequalities of
(34) one obtains 
\begin{equation}
0\leq\frac{1}{n+1}\sum_{k=1}^{n}k^{2}\left|\gamma_{k}^{(m)}-\frac{1}{k}\right|^{2}<2\epsilon
\end{equation}
whenever $m,n>N_{1}(\epsilon)$ . So, for $m\in\mathbb{N}$, consider
the functions $g_{m}$, $F_{m}$ defined by $g_{m}(r)+2\mathit{log}(1-r)=\sum_{n=1}^{\infty}\lambda_{n}^{(m)}r^{n}$
where $\lambda_{k}^{(m)}=2(\gamma_{n}^{(m)}-n^{-1})$ and $F_{m}(r)=\mathit{exp\,}(g_{m}(r)+2\mathit{log}(1-r))=r^{-1}(1-r)^{2}f_{m}(r)=\sum_{n=0}^{\infty}b_{n}^{(m)}r^{n}$.
Then $nb_{n}^{(m)}=\sum_{k=1}^{n}\lambda_{k}^{(m)}b_{n-k}^{(m)}$
for $m\in\mathbb{N}$, and inductively it follows that
\begin{equation}
\frac{1}{n+1}\sum_{k=1}^{n}kb_{k}^{(m)}=\frac{1}{n+1}\sum_{j=1}^{n}j\lambda_{j}^{(m)}s_{n-j}^{(m)}
\end{equation}
where
\begin{equation}
s_{k}^{(m)}=\sum_{j=0}^{k}b_{j}^{(m)}=a_{k+1}^{(m)}-a_{k}^{(m)}\qquad m\in\mathbb{N}.
\end{equation}
The Cauchy-Schwarz inequality applied to (36) yields
\begin{equation}
\left|\frac{1}{n+1}\sum_{k=1}^{n}kb_{k}^{(m)}\right|^{2}\leq\left(\frac{4}{n+1}\sum_{k=1}^{n}k^{2}\left|\gamma_{k}^{(m)}-\frac{1}{k}\right|^{2}\right)\left(\frac{1}{n+1}\sum_{k=0}^{n-1}\left|s_{k}^{(m)}\right|^{2}\right).
\end{equation}
But (32) and \cite{1}, Theorem 5.10 imply that $|s_{n}^{(m)}|^{2}\leq(\alpha-\epsilon)^{-1}\mathit{exp\,}(2\delta)$
if $m>N_{1}(\epsilon)$, where $\delta$ denotes Milin's constant.
Hence, if $\delta_{n}^{(m)}$, $m,n\in\mathbb{N}$ is defined by $\delta_{n}^{(m)}=(n+1)^{-1}\sum_{k=1}^{n}kb_{k}^{(m)}$,
(35) and (38) imply that 
\begin{equation}
\left|\delta_{n}^{(m)}\right|=\left|\frac{1}{n+1}\sum_{k=1}^{n}kb_{k}^{(m)}\right|<\sqrt{8\epsilon(\alpha-\epsilon)^{-1}\mathit{exp}\,(2\delta)}
\end{equation}
whenever $m,n>M(\epsilon)$ where $M(\epsilon)=\mathit{max}\{N_{1}(\epsilon),N_{2}(\epsilon)\}$.
From now on the lines of proof essentially follow Tauber's well-known
proof of his second theorem, however adapted to sequences of functions.
Tauber(\cite{8}, p.276 or \cite{9}, chapter 7) obtains the two formulas
\[
\sum_{k=1}^{n}b_{k}^{(m)}=\sum_{k=1}^{n}\left(1+\frac{1}{k}\right)\delta_{k}^{(m)}-\sum_{k=1}^{n}\delta_{k-1}^{(m)}=\left(1+\frac{1}{n}\right)\delta_{n}^{(m)}+\sum_{k=1}^{n-1}\frac{1}{k}\delta_{k}^{(m)}
\]
and
\[
F_{m}(r)-b_{0}^{(m)}=\sum_{n=1}^{\infty}b_{n}^{(m)}r^{n}=\sum_{k=1}^{\infty}\frac{\delta_{k}^{(m)}}{k}r^{k}+(1-r)\sum_{k=1}^{\infty}\delta_{k}^{(m)}r^{k}
\]
which hold if $n\geq2$ and $r\in(0,1)$. The first formula easily
can be verified inductively and the second follows by a straightforward
calculation. By subtracting the last two equations it follows that
\begin{equation}
F_{m}(r)-s_{n}^{(m)}=(1-r)\sum_{k=1}^{\infty}\delta_{k}^{(m)}r^{k}+\sum_{k=n}^{\infty}\frac{\delta_{k}^{(m)}}{k}r^{k}+\sum_{k=1}^{n-1}\frac{\delta_{k}^{(m)}}{k}\left(r^{k}-1\right)-\left(1+\frac{1}{n}\right)\delta_{n}^{(m)}
\end{equation}
which holds if $m\in\mathbb{N}$, $n\geq2$ and if $r\in(0,1)$. The
Cesaro-means $\sigma_{n}^{(m)}$, $m\in\mathbb{N}$, $n\in\mathbb{N}_{0}$
are defined by $\sigma_{n}^{(m)}=(n+1)^{-1}\sum_{k=0}^{n}s_{k}^{(m)}$,
hence $\delta_{n}^{(m)}=s_{n}^{(m)}-\sigma_{n}^{(m)}$ and therefore
(40) can be written in the form
\begin{equation}
F_{m}(r)-\sigma_{n}^{(m)}=(1-r)\sum_{k=1}^{\infty}\delta_{k}^{(m)}r^{k}+\sum_{k=n}^{\infty}\frac{\delta_{k}^{(m)}}{k}r^{k}+\sum_{k=1}^{n-1}\frac{\delta_{k}^{(m)}}{k}\left(r^{k}-1\right)-\frac{\delta_{n}^{(m)}}{n}
\end{equation}
whenever $n\geq2$ and $r\in(0,1)$. In order to simplify notation
let $\varepsilon=\sqrt{8\epsilon(\alpha-\epsilon)^{-1}\mathit{exp}\,(2\delta)}$
and observe that by (39), there exists a constant $L>0$ such that
$|\delta_{n}^{(m)}|<L$ if $m,n\in\mathbb{N}$. Then, since $|r^{k}-1|\leq k(1-r)$
for $k\in\mathbb{N}$, (39) and (41) yield
\begin{eqnarray}
\left|F_{m}(r)-\sigma_{n}^{(m)}\right| & \leq & 2(1-r)\sum_{k=1}^{M(\epsilon)}\left|\delta_{k}^{(m)}\right|r^{k}+2(1-r)\sum_{k=M(\epsilon)+1}^{\infty}\left|\delta_{k}^{(m)}\right|r^{k}+\sum_{k=n}^{\infty}\frac{\left|\delta_{k}^{(m)}\right|}{k}r^{k}+\frac{\left|\delta_{n}^{(m)}\right|}{n}\nonumber \\
 & < & 2(1-r)M(\epsilon)L+2\varepsilon+\frac{1}{1-r}\frac{\varepsilon}{n}+\frac{\varepsilon}{n}
\end{eqnarray}
if $m,n>\mathit{M(\epsilon)}$ and $r\in(0,1)$. Now let $F:[0,1]\rightarrow\mathbb{R}$
be defined by $F(r)=(1-r)^{2}r^{-1}f(r)$ and remember that $F_{m}(r)=(1-r)^{2}r^{-1}f_{m}(r)$
for $m\in\mathbb{N}$. Then hypothesis (ii) and Lemma 2 imply that
the sequence $(|F_{m}|)$ converges uniformly in $[0,1]$ to $|F|$
as $m\rightarrow\infty$, that is there exists a $N_{3}(\epsilon)\in\mathbb{N}$
such that $\left||F_{m}(r)|-|F(r)|\right|<\epsilon$ whenever $m>N_{3}(\epsilon)$.
But since $|F|$ is continuous in $[0,1]$ there exists a $\rho>0$
such that $\left|\left|F(r)\right|-\alpha\right||<\epsilon$ if $1-r<\rho$.
By (32) it follows that $\left|\alpha-\alpha_{m}\right|<\epsilon$
if $m>N_{1}(\epsilon)$ and hence
\begin{equation}
\left||F_{m}(r)|-\alpha_{m}\right|\leq\left||F_{m}(r)|-|F(r)|\right|+\left|\left|F(r)\right|-\alpha\right|+\left|\alpha-\alpha_{m}\right|<3\epsilon
\end{equation}
whenever $m>\mathit{max}\{N_{1}(\epsilon),N_{3}(\epsilon)\}$ and
$1-r<\rho$. Finally, in order to complete the proof, let $r=1-n^{-1}$
and observe that there exists a $N_{4}(\epsilon)\in\mathbb{N}$ such
that $1-r=n^{-1}<\rho$ if $n>N_{4}(\epsilon)$. Similarly, if $1-r=n^{-1}$,
there exists a $N_{5}(\epsilon)\in\mathbb{N}$ such that 
\begin{equation}
(1-r)M(\epsilon)L=\frac{M(\epsilon)L}{n}<\epsilon
\end{equation}
if $n>N_{5}(\epsilon)$. Observe that by (37) $\sigma_{n}^{(m)}=(n+1)^{-1}\sum_{k=0}^{n}s_{k}^{(m)}=(n+1)^{-1}a_{n+1}^{(m)}$
for $m,n\in\mathbb{N}$ and that therefore 
\[
\left|\alpha_{m}-\left|\frac{a_{n+1}^{(m)}}{n+1}\right|\right|=\left|\alpha_{m}-\left|\sigma_{n}^{(m)}\right|\right|\leq\left|\left|F_{m}(1-\frac{1}{n})\right|-\left|\sigma_{n}^{(m)}\right|\right|+\left|\alpha_{m}-\left|F_{m}(1-\frac{1}{n})\right|\right|<(2\epsilon+4\varepsilon)+3\epsilon
\]
by (42), (43) and (44) whenever $m,n>\mathit{max}\{N_{1}(\epsilon),N_{2}(\epsilon),N_{3}(\epsilon),N_{4}(\epsilon),N_{5}(\epsilon)\}$.
This completes the proof.
\end{proof}
Several questions arise, the most interesting perhaps whether Theorem
2 remains true if the function $g:\Delta\rightarrow\mathbb{C}$ defined
by $g(z)=f(z^{-1})^{-1}$ is not a full mapping(where $f\in S$ denotes
the limit function of Theorem 2). The proof of Theorem 2 suggests
that this might not be true since for this class of functions strict
inequality holds in Basilevich's theorem. The observations made so
far suggest the following definition of an approximation measure for
schlicht functions.
\begin{defn}
A schlicht function $f\in S$ will be called not badly approximable
if for any sequence $(f_{n})$, $f_{n}\in S$ such that $f_{n}\rightarrow f$
locally uniformly in $\mathbb{D}$ and $\alpha(f_{n})\rightarrow\alpha(f)$
as $n\rightarrow\infty$ and for any $\epsilon>0$ there exists a
number $N\in\mathbb{N}$(dependent only on $\epsilon$ and the sequence
$(f_{n})$) so that $\left|k^{-1}a_{k}^{(n)}-\alpha(f)\right|<\epsilon$
whenever $k,n>N$. Here it is assumed that $f_{n}(z)=\sum_{k=1}^{\infty}a_{k}^{(n)}z^{k}$
for $z\in\mathbb{D}$ and $n\in\mathbb{N}$.
\end{defn}
In the terminology of Definition 1 every schlicht function $f\in S$
of slow growth type is not badly approximable by Theorem 1. By Theorem
2 every schlicht function $f\in S$ whose associated inverted function
$g:\Delta\rightarrow\mathbb{C}$ defined by $g(z)=f(z^{-1})^{-1}$
is a full mapping is not badly approximable either. The full-mapping
property and boundedness of the image regions for instance are geometric
properties of the image regions of schlicht functions and so Theorems
1 and 2 also allow for a geometric interpretation.\\
The use of approximation measures has a long-standing tradition in
the theory of diophantine approximation and the current paper was
actually inspired by that. To mention just one of the approximation
measures of diophantine approximation, an irrational number $\lambda$
is called badly approximable if and only if there is a constant $c=c(\lambda)>0$
such that $|\lambda-p/q|>c/q^{2}$ for every rational number $p/q$(see
\cite{10}, chapter I.5). There are continuum many badly approximable
irrationals and continuum many not badly approximable irrationals
and this particular approximation measure has a close connection to
the continued fraction expansion of an irrational. If there were badly
approximable functions the topology and distribution of power series
would show phenomena similar to the topology and distribution of numbers.
However, regardless whether there are badly approximable schlicht
functions or not Theorems 1 and 2 provide a refined picture of the
topology and distribution of schlicht functions as their application
to asymptotic extremal problems(examples 1 and 2) shows.

\end{document}